\documentclass[10]{amsart}

\usepackage{graphicx}
\usepackage{enumitem}
\usepackage{color}

\newtheorem{theorem}{Theorem}[section]
\newtheorem{lemma}[theorem]{Lemma}
\newtheorem{corollary}[theorem]{Corollary}
\theoremstyle{definition}

\theoremstyle{remark}
\newtheorem{remark}[theorem]{Remark}

\numberwithin{equation}{section}

\newcommand{\F}{{\mathbb{F}_q}}
\newcommand{\PG}{\mathrm{PG}(2,q)}
\newcommand{\C}{\mathcal{C}}
\newcommand{\m}{(\mathrm{mod}\;4)}

\begin{document}

\title{On planar arcs of size $(q+3)/2$}

\author{G\"{u}lizar G\"{u}nay$^*$}
\address{G\"{u}lizar G\"{u}nay, Faculty of Engineering and Natural Sciences, Sabanc{\i} University, Tuzla, İstanbul 34956, Turkey}
\email{gunaygulizar@sabanciuniv.edu}

\author{Michel Lavrauw}
\address{Michel Lavrauw, Faculty of Engineering and Natural Sciences, Sabanc{\i} University, Tuzla, İstanbul 34956, Turkey}
\email{mlavrauw@sabanciuniv.edu}

\thanks{The authors were supported by {\em The Scientific and Technological Research Council of Turkey}, T\"UB\.{I}TAK (project no. 118F159).
\\
$^*$ Corresponding author}

\begin{abstract}
The subject of this paper is the study of small complete arcs in $\PG$, for $q$ odd, with at least $(q+1)/2$ points on a conic. We give a short comprehensive proof of the completeness problem left open by  Segre in his seminal work ~\cite{Segre1967}. This gives an alternative to Pellegrino's long proof which was obtained in a series of papers in the 1980s. As a corollary of our analysis, we obtain a counterexample to a misconception in the literature \cite{Hirschfeld_review}.
\end{abstract}
\maketitle
\section{Introduction and motivation}
\par Let $\F$ denote the finite field with $q$ elements, and $\PG$ denote the projective plane over $\F$. An \emph{arc} in $\PG$ is a set of points no three of which are collinear. An arc $K$ is \emph{complete} if $K$ can not be extended to a larger arc. Arcs in $\PG$ are also known as {\em planar arcs} and are the two-dimensional version of arcs in projective spaces of any dimension $\geq 2$. Given an arc $K$ we say that a point of $\PG \setminus K$ is {\em $K$-covered} if it lies on a secant of $K$ and {\em $K$-free} otherwise.

\bigskip

\par It is well known that arcs are equivalent to MDS codes and there is a vast literature on the subject. As we aim to keep this paper short we refer to \cite[Chapter 13]{HiKoTo_book} and the recent expository paper \cite{MR4073887} for further details and references.

\bigskip

\par An easy geometric argument allows one to show that an arc can have size at most $q+1$ if $q$ is odd and $q+2$ if $q$ is even. In $1955$, Beniamino Segre classified arcs of size $q+1$ as conics when $q$ is odd. For $q$ even the classification of arcs of size $q+2$ (called {\em hyperovals}) remains a challenging open problem. Segre's initial studies on arcs naturally lead to the following questions:
\begin{enumerate}
\item What are the possible sizes of complete arcs in $\PG$?
\item How many points can a conic have in common with a complete arc which is not a conic?
\end{enumerate}
These questions have been extensively studied and many mathematicians have contributed to the large variety of constructions of complete arcs.
In many of these constructions, a large portion of the points of the arc is chosen among the points of a conic or a cubic curve, following up on ideas from Segre~\cite{Segre1962} and Lombardo-Radice~\cite{Lombardo1956}.

\par One of the first such constructions for $q$ odd, suggested by Segre, gives rise to arcs containing roughly half of the points of a conic. The construction starts from a conic $\C$ and one point $R$ outside $\C$. Depending on whether $R$ is external (on two tangents) or internal (on no tangents) to $\C$ one obtains an arc of size $(q+5)/2$ or $(q+3)/2$ by choosing one point on each of the secants through $R$ (and including the points of tangency in the case that $R$ is external). This clearly gives an arc $K$, but it remains to determine whether the arc $K$ is complete or not.
Also, if $K$ is not complete, then the following natural problem is to decide which points (how many) can be added to complete $K$.
In order to prove the completeness of an arc, there are some methods known, relying on results from group theory, algebraic geometry, or Galois fields, some of which were used to prove the results listed below.

\bigskip

\par The next paragraphs describe some more details concerning the complete arcs obtained from Segre's construction.

\bigskip

\par Assume that $R$ is an external point to the conic $\C$. If $q \equiv 3\;\m$, then completeness of the arc $K$ of size $(q+5)/2$ obtained from the above mentioned construction, was proved by Segre in~\cite[p.~152]{Segre1967}. In $1992$, Pellegrino proved in~\cite{Pellegrino1992} that if $K$ is a complete arc, for $q$ odd, containing $(q+3)/2$ points from an irreducible conic $\C$ of $\PG$, then the possible sizes for $K$ are $(q+5)/2$ and $(q+7)/2$ in general, with the additional values $(q+9)/2$ for $q \equiv 3\;\m$. Later, in $2001$, using linear collineations, Korchm\'{a}ros and Sonnino gave an alternative proof to Pellegrino's result that in $\PG$, with $(q+1)/2$ an odd prime, every arc sharing $(q+3)/2$ points with a conic contains at most four points outside the conic. This number reduces to two, when, in addition $q^2\equiv 1 \;(\hbox{mod}\;16)$ is assumed~\cite[Theorem 1.1]{KoSo2003}.

\bigskip

\par In the case that $R$ is an internal point to the conic $\C$, the completeness of $K$ was left as an open problem by Segre in \cite{Segre1967}. In this paper, we give a short proof for the solution to this problem.
Our proof gives an alternative to the proofs contained in a series of papers by Pellegrino from the $1980$'s (see~\cite{Pellegrino1981, Pellegrino1982, Pellegrino1982_UMI, Pellegrino1981_UMI}.) The papers are written in Italian and some of them we found difficult to get hold of. In $1981$ Pellegrino constructed examples of complete arcs of size $(q+3)/2$ in $\PG$ for $q \equiv 3\;\m$ in~\cite{Pellegrino1981}. One year later, he constructed complete arcs of size $(q+3)/2$ in $\PG$ for $q \equiv 1\;\m$ in~\cite{Pellegrino1982}. The method used in the papers relies on results in finite fields from~\cite{Pellegrino1982_UMI} and~\cite{Pellegrino1981_UMI}, whereas our proof uses bounds for the number of points on algebraic curves over a finite field.

\bigskip

\par With this paper we aim to provide comprehensive proof of the problem left open by Segre in \cite{Segre1967}. As a corollary of our proof, we obtain examples of arcs which give counterexamples to the statement in~\cite{Hirschfeld_review}. This concerns the existence of a line satisfying the hypothesis for the main theorem from \cite{Pellegrino1993} in which Pellegrino studied the completion of arcs sharing (any) $(q+1)/2$ points with a conic but with an extra assumption.

\bigskip

\par The paper is organized as follows. Section $2$ contains preliminaries and Section $3$ contains the main results of the paper. In Section $4$, we give a counterexample to the claim in~\cite{Hirschfeld_review}. Using the same construction in the paper, for $q=9$, $11$, $13$ sizes of complete arcs are given in Section $5$.

\section{Preliminaries}
\par In this section we collect the necessary background for the proof of the completeness of the arc $K=H\cup \{R\}$ of size $(q+3)/2$ where $H$ is contained in a conic $\C$ and consists of one point on each of the secants of $\C$ through $R$.

\bigskip

\par An algebraic curve $\mathcal{F}$ in $\PG$ is equal to the zero locus ${\mathcal{Z}}(F)$ of a homogeneous polynomial $F \in \F[X,Y,Z]$. The curve $\mathcal{F}$ is $\emph{irreducible}$ if $F$ is irreducible over $\F$; $\mathcal{F}$ is $\textit{absolutely\;irreducible}$ if $F$ is irreducible over $\overline{\mathbb{F}}_q$, the algebraic closure of $\F$. If $P=\langle U\rangle$ is a point of the irreducible curve $\mathcal{F}=\mathcal{Z}(F)$ of degree $d$ and $\ell=\langle U, V\rangle$ is a line not contained in $\mathcal{F}$, then the $\emph{intersection multiplicity}$ $m_P(\ell, \mathcal{F})$ is the multiplicity of $t=0$ in $F(U+tV)$. The $\emph{multiplicity}$ \label{multiplicity} of $P$ on $\mathcal{F}$, denoted by $m_P(\mathcal{F})$, is the minimum of $m_P(\ell, \mathcal{F})$ for all lines $\ell$ through $P$. Then $P$ is a $\emph{singular}$ point of $\mathcal{F}$ if $m_P(\mathcal{F})> 1$. A line $\ell$ is a $\emph{tangent line}$ to $\mathcal{F}$ at $P$ if $m_P(\ell,\mathcal{F})>m_P(\mathcal{F})$. The point $P$ is called an $\emph{ordinary singular point}$ if $\mathcal F$ has $m_P(\mathcal{F})$ distinct tangents at the point $P$.
The set of $\emph{$\F$-rational\;points}$ of the curve $\mathcal{F}$ is defined as the set of points of $\PG$ where $F$ vanishes.

\begin{lemma}\cite[Lemma $8$]{SeBa1971}\label{SeBa1971} Let $\mathcal{F}$ be a projective plane curve of degree $k$ defined over an arbitrary field $E$. If there exists a point $P$ of $\mathcal{F}$ and a tangent line $\ell$ at $P$, for which
\begin{enumerate}[label=(\roman*)]
  \item  $\ell$ counts once among the tangents of $\mathcal{F}$ at $P$,
  \item  the intersection multiplicity of $\mathcal{F}$ and $\ell$ at $P$ is equal to $k$, and
  \item $\mathcal{F}$ has no linear component through $P$,
\end{enumerate}
then $\mathcal{F}$ is absolutely irreducible.
\end{lemma}

\par The following theorem is a special case of the Hasse-Weil bound concerning the number of $\F$-rational points lying on an absolutely irreducible projective curve $\mathcal{F}$ with genus $g$.

\begin{theorem}\cite{Weil1945}\label{Weil1945} If $\mathcal{F}$ is an absolutely irreducible nonsingular projective curve in $\PG$ of degree $d$, and $N$ denotes the number of $\F$-rational points of $\mathcal{F}$, then \[|q+1-N| \leq  (d-1)(d-2)\sqrt{q}.\]
\end{theorem}

\section{Completeness proof}

From now on $q$ will always be odd. We denote the set of squares in $\F$ by $\square$ and the set of non-squares in $\F$ by $\triangle$.
The following lemma gives an algebraic way to distinguish between internal and external points of a conic in $\PG$.

\begin{lemma} If $\C=\mathcal{Z}(f)$ is a conic in $\PG$, $q$ odd, and $f(P)\in \square$ ($\triangle$) for an external point $P$, then
a point $Q\notin \C$ is external (internal) if $f(Q)\in \square$ and internal (external) if $f(Q) \in \triangle$.
\end{lemma}

\par We fix $\C$ to be the conic $\C=\mathcal{Z}(XY-Z^2)$,
and we define the set $H\subseteq \C$ as
$$H=\{(1,s^4,s^2)\; |\; s \in \F\}$$ and $H'=\C\setminus H$.
Our aim is to determine the set of $H$-free points in $\PG$.

\bigskip

\par For the point
$$R(a,b,c)\notin \C$$
we define an involution $\tau_R$ on the points of $\C$ as follows. If $P(1,s^2,s)$ is a point of $\C$, then $\tau_{R}(P)=Q(1,t^2,t)$ is the second intersection of the line $PR$ with $\C$. If $PR$ is a tangent line, then we define $\tau_R(P)=P$. The collinearity of the points $P$, $Q$ and $R$ gives the equation \[ast+b-c(s+t)=0.\] Therefore, $t=\frac{c s-b}{a s-c}$.
The stabiliser of $\C$ inside PGL$(3,q)$ is isomorphic to PGL$(2,q)$, and under a suitable isomorphism
this involution corresponds to the projectivity $\varphi_R$ of PG$(1,q)$ with matrix \[M_R=\left(
               \begin{array}{cc}
                 c & -b \\
                 a & -c \\
               \end{array}
             \right).\]
Notice that $-c^2+ba \neq 0$, since $R$ does not belong to the conic $\C$.
Parameterising the projective line as the set $\F\cup\{\infty\}$. We consider $\infty$ as a non-square, i.e. $\infty\in \triangle$. We obtain $\varphi_R(s)=\frac{cs-b}{as-c}$ for $as-c\neq 0$, $\varphi_R(\infty)=ca^{-1}$ and $\varphi_R(ca^{-1})=\infty$. This leads to the following characterisation of $H$-covered points.
\begin{lemma}
The point $R$ is $H$-covered if and only if there
exist some $x\in \square$ for which $\varphi_R(x)\in \square$.
\end{lemma}
\begin{proof}
Observing that $\infty$ corresponds to the point $(0,0,1)$ which does not belong to $H$, the proof follows from the above.
\end{proof}

\subsection{} First we focus on the case $abc\neq 0$.

\begin{lemma}\label{involution}
If $q\geq 17$ and the point $R(a,b,c)$, where $abc\neq0$, does not belong to $\C$ then $R(a,b,c)$ is $H$-covered.
\end{lemma}
\begin{proof}
To the point $R(a,b,c)$ we associate the affine algebraic curve $\mathcal{F}_{a,b,c}=\mathcal{Z}(f)$ where \[f(X,Y)=(c X^2-b)-\mu Y^2(a X^2-c)\;\mbox{with}\;\mu\in\triangle.\]
We will show that the hypothesis that $R(a,b,c)$ is $H$-free leads to a contradiction with the number of points on the curve $\mathcal F_{a,b,c}$ whenever $q\geq 17$.
We begin by proving that $\mathcal{F}_{a,b,c}$ satisfies the conditions of the irreducibility criterion given in Lemma~\ref{SeBa1971}.

\bigskip

(i) First of all, one easily verifies that $\mathcal{F}_{a,b,c}$ has no affine singular points by setting the partial derivatives of $f(X,Y)$ evaluated at an affine point $(x,y)$ on $\mathcal{F}_{a,b,c}$ equal to zero.

\bigskip

(ii) Next we show that the ideal points $X_\infty(1,0,0)$ and $Y_\infty(0,1,0)$ of the projective closure $\mathcal{F}$ of $\mathcal{F}_{a,b,c}$ are ordinary singularities of multiplicity $2$.
The curve $\mathcal{F}$ is equal to the zero locus of $F$ where
\begin{eqnarray}\label{thecurveg}
F(X,Y,Z)=(c X^2Z^2-b Z^4)-\mu Y^2(a X^2-c Z^2).
\end{eqnarray}
It is straightforward to verify that the points $X_\infty$ and $Y_\infty$ are both double points of $\mathcal F$, and by B\'ezout's theorem, they are the only points of $\mathcal{F}$ on the line $Z=0$.

\par Since the line with equation $Z=0$ is not a tangent line of $\mathcal{F}$, the tangents through $X_\infty$ have an equation of the form $Y-eZ=0$.
Substituting the point $(x,e)$ into the equation of the affine curve, we obtain
\begin{eqnarray*}
 \nonumber
 f(x,e)&=& (c x^2-b)-\mu e^2(a x^2-c)=x^2(c-\mu a e^2)-b+\mu c e^2=0.
\end{eqnarray*}
If $c=\mu e^2a$ and $f(x,e)=0$, then $f(x,e)=-b a + c^2=0$, contradicting the fact that $R(a,b,c)$ does not belong to the conic $\C$. It follows that the tangents of $\mathcal{F}$ at $X_\infty$ are exactly the lines $Y=e$ for which $e^2= \mu^{-1} a^{-1}c$. Since there are $2$ distinct tangents at $X_\infty$ the point $X_\infty$ is an ordinary singularity. Similarly one verifies that also the point $Y_\infty$ is an ordinary singular point with multiplicity $2$.

\bigskip

(iii) From the definition of intersection multiplicity for $X_\infty(1,0,0)$ and $V(0,e,1)$, where $e^2= \mu^{-1} a^{-1}c$, we obtain that
\begin{eqnarray*}
\nonumber
h(t) &=& F((1,0,0)+t(0,e,1))=F(1,et,t)= t^4(-b+\mu c e^2).
\end{eqnarray*}
The multiplictiy of the root $t=0$, and therefore, the intersection multiplicity of the curve $\mathcal{F}$ and the tangent line $Y=e$, in which, $e^2= \mu^{-1} a^{-1}c$, is $4$.

\bigskip

(iv) Finally, it easily follows that $\mathcal F$ has no linear component through $X_\infty$, since a linear component would have to intersect the remaining component(s) of the curve in at least three singular points, counted with multiplicity. Since $X_\infty$ has multiplicity 2, this contradicts (i).

\bigskip

\par It follows from Lemma~\ref{SeBa1971} that  $\mathcal{F}$ is absolutely irreducible. Hence, we can use the Hasse-Weil theorem
\[|N-q+1|\leq 2g\sqrt{q}\;\;\Rightarrow\;\;q+1-2g\sqrt{q}\leq N \leq q+1+2g\sqrt{q}\] where $g$ denotes the genus of the curve $\mathcal{F}$. A projective irreducible curve with no affine singularities has genus \[g=\binom{d-1}{2}-\sum_{P\in Sing(\mathcal{F})}\binom{m_P}{2}\] where $d$ denotes the degree of the curve and $m_p$ denotes the multiplicity of the singular point $P$ of the projective curve. For the curve $\mathcal{F}$, we have $2$ singular points $X_\infty$ and $Y_\infty$ of multiplicities $2$. Therefore $g=1$.

\bigskip

\par Assume that $R(a,b,c)$ is $H$-free. Fix $\mu \in \triangle$. For each $x^2\in \square$, we have \[\frac{c x^2-b}{a x^2-c}=\mu y^2\;\mbox{for\;some}\;y\in \F.\]

We will show that this leads to a contradiction unless $q<17$, by further analysing the curve $\mathcal{F}=\mathcal{Z}(F)$ as defined in \ref{thecurveg}.

(i) Clearly $X_\infty$ and $Y_\infty$ are the only two points of the form $(x,y,0)$ on $\mathcal{F}=\mathcal{Z}(F)$.

(ii) A point of the form $(x,0,1)$ lies on the curve $\mathcal{F}$ if and only if $ x^2=c^{-1}b$ and the point of the form $(0,y,1)$ lies on the curve $\mathcal{F}$ if and only if $y^2=\mu^{-1}c^{-1}b$.
Therefore, the points of these forms can not lie on the curve $\mathcal{F}$ at the same time. It follows that the curve $\mathcal{F}$ has $2$ points of the form $(x,0,1)$ or $(0,y,1)$.

(iii) If the point $(x,y,1)$ lies on the curve $\mathcal{F}$ then
$\mu y^2(a x^2-c)=c x^2-b$.
By the hypothesis, for each $x^2\notin\{\frac{c}{a},\frac{b}{c}\}$ we obtain
two such points on $\mathcal{F}$.
If $x^2=\frac{c}{a}$, then $c^2-ba=0$, contradicting the fact that $R(a,b,c)$ is not on the conic $\mathcal{C}=\mathcal{Z}(XY-Z^2)$.
If $x^2\neq\frac{c}{a}$ but $x^2=\frac{b}{c}$ then $ y=0$ and this case is already accounted for in $(ii)$ above.
So excluding the previous cases $(i)$ and $(ii)$ we obtain at least $2(q-5)$ points on the curve $\mathcal{F}$ of the form $(x,y,1)$ with $x,y\neq 0$.

We conclude that the number of $\F$-rational points on the curve $\mathcal{F}$ is at least $2q-6$. The Hasse-Weil bound gives
$2q-6 \leq q+1+2\sqrt{q}$, which implies $q\leq 13$.
\end{proof}

\begin{remark} \label{remark} The fact that the curve $\mathcal{F}_{a,b,c}$ has an absolutely irreducible $\F$-rational component can also
be deduced from the fact that the curve $\mathcal{F}_{a,b,c}=\mathcal{Z}(f)$ where $f(X,Y)=(c X^2-b)-\mu Y^2(a X^2-c)$ is a Kummer extension of the curve $\mathcal{H} =\mathcal{Z}(h')$ where $h'(X) =(c X^2-b)-\mu u(a X^2-c)$ for $Y^2=u$.
\end{remark}

\subsection{} If $abc=0$, then we have the following cases.

\bigskip

\subsubsection{}\label{(i)} If $c=0$, then $ab\neq 0$ and $\varphi_R$ sends $x$ to $y=-(ba^{-1})x^{-1}$. If $ba^{-1}\in\square$, then the point $R(a,b,0)$ is $H$-free for $q\equiv 3\;\m$. Similarly, if $ba^{-1}\in\triangle$, then the point $R(a,b,0)$ is $H$-free for $q\equiv 1\;\m$.

\bigskip

\subsubsection{}\label{(ii)} If both $b$ and $a$ are $0$, then $\varphi_R$ sends $x\in \square$ to $\varphi_R(x)\in \square$ for $q\equiv 1\;\m$ and it sends $x\in \triangle$ to $\varphi_R(x)\in \square$ for $q\equiv 3\;\m$. Therefore, if $q\equiv1\;\m$, then $R(0,0,1)$ is $H$-covered, and if $q\equiv3\;\m$, then the point $R(0,0,1)$ is $H$-free.

\bigskip

\subsubsection{}\label{(iii)} Next consider the cases $b=0$, $ac\neq 0$ and $a=0$, $bc\neq 0$.

\begin{lemma}\label{Ex:cases} If $q\geq 17$, then the points $U(0,b,c)$ and $V(a,0,c)$ are $H$-covered for $bc\neq 0$ and $ac\neq0$.
\end{lemma}
\begin{proof} Assume that the point $U(0,b,c)$ is $H$-free. Then it can not be on chords of $H$ which consist of the lines
 \[\ell_k\;:\;Y=\alpha^{2k}Z,\;\mbox{where}\;k=1,2,.., (q-1)/2\] which pass through the points $(1,0,0)$ and $(1,\alpha^{4k},\alpha^{2k})$, and \[\ell_{ij}\;:\;Y=(\alpha^{2i}+\alpha^{2j})Z-(\alpha^{2j})(\alpha^{2i})X,\;\mbox{where}\;i,j=1,2,.., (q-1)/2,\;i\neq j\] which pass through the points $(1,\alpha^{4i},\alpha^{2i})$ and $(1,\alpha^{4j},\alpha^{2j})$, where $\mathbb{F}_q^*=\langle\alpha\rangle$. Therefore, for the point $U(0,b,c)$, this implies \[b\neq \alpha^{2k}c\;\mbox{and}\;bc^{-1}\neq\alpha^{2i}+\alpha^{2j}\;\hbox{where}\;i,j,k=1,..,(q-1)/2.\]
This contradicts the fact that a non-square element $bc^{-1}$ in $\F$ can be written as the sum of two nonzero square elements. We conclude that the point $U(0,b,c)$ is $H$-covered.

\bigskip

\par From the above it also follows that the point $V(a,0,c)$ lies on the chord $\ell_{ij}$ if and only if \[(\alpha^{2i}+\alpha^{2j})c-(\alpha^{2j})(\alpha^{2i})a=0,\;\mbox{where}\;i,j=1,2,.., (q-1)/2,\;i\neq j.\]
For $\alpha^{2i}=X$, $\alpha^{2j}=Y$ and $a^{-1}c=\mu'$, we obtain the curve $\mathcal{Z}(f')$ with \[f'(X,Y)= X^2Y^2-\mu'(X^2+Y^2)=0\;\mbox{with}\;\mu'\neq 0.\]
Therefore the point $V(a,0,c)$ is $H$-covered if the curve $\mathcal{Z}(f')$ has points with coordinates in $\F$.
Straightforward computations show that the points $(0,0,1)$, $(0,1,0)$ and $(1,0,0)$ are the singular points of
the projective closure $\mathcal{F}=\mathcal{Z}(F')$ of $\mathcal{Z}(f')$, where
\[F'(X,Y,Z)=X^2Y^2-\mu' Z^2(X^2 + Y^2).\]
The multiplicity of the point $(1,0,0)$ on $\mathcal{F}$ is $2$, since the multiplicity of $t=0$ in $\mathcal{F}(1,t,0)= t^{2}$ is $2$. By symmetry, also the point $(0,1,0)$ has multiplicity $2$. Similarly, one verifies that the point $(0,0,1)$ has multiplicity $2$ on the curve $\mathcal{F}$.

\bigskip

\par Now we want to prove that the curve $\mathcal{F}$ is absolutely irreducible. Since the degree of the curve is $4$, the degrees of the irreducible factors of the polynomial defining the curve are either $1$ and $3$ or $2$ and $2$.

\bigskip

\par The curve $\mathcal{F}$ has no linear component $\mathcal{Z}(Y-dX)$ through $(0,0,1)$
since \[f'(x,dx)=d^2x^4-\mu'(x^2+d^2x^2)=x^2(d^2x^2-\mu'(1+d^2))\] is not zero for all $x\in \F$.
Similarly, since
\[f'(x,e)=x^2e^2-\mu'(x^2+e^2)=x^2(e^2-\mu')-\mu'e^2\]
is not zero for all $x\in \F$, $\mathcal{F}$ has no linear component of the form $\mathcal{Z}(Y-eZ)$.

By symmetry, the curve $\mathcal{F}$ has no component of the form $\mathcal{Z}(X-fZ)$, either. Since a linear component must pass through one of the three singular points, this shows that the curve $\mathcal{F}$ has no linear component.

\bigskip

\par If we assume that the curve $\mathcal{F}$ has two irreducible factors of degrees $2$, then there are $2$ polynomials $f_1\neq 0$ and $f_2\neq 0$ such that $\mathcal{F}=f_1f_2$ with $\mathrm{deg}(f_1)=2$ and $\mathrm{deg}(f_2)=2$. It follows from the classification of pencils of conics in $\PG$, $q$ odd, that the pencil $\mathcal{P}(\mathcal{Z}(f_1),\mathcal{Z}(f_2))$ is equivalent to the pencil of type $o_{13}$, i.e. $\mathcal{P}(2XY,Y^2-Z^2)$, corresponding to the second column of Table $5$ of \cite{MR4045559}, since this is the only pencil of conics with $3$ base points (which are $(1,0,0)$, $(0,1,1)$ and $(0,-1,1)$). After a suitable coordinate transformation, we may therefore assume that $\mathcal{Z}(f_1)$ corresponds to $C_1=\mathcal{Z}(2XY+Y^2-Z^2)$ and $\mathcal{Z}(f_2)$ corresponds to $C_2=\mathcal{Z}(2\gamma XY+Y^2-Z^2)$ for some $\gamma \in \F \setminus \{0,1\}$. Computing the tangent lines of these conics at the base points of the pencil, we find that $\mathcal{Z}(Y)$ is a common tangent line of the conics at $(1,0,0)$ and there are $2$ different tangent lines at each of the points $(0,1,1)$ and $(0,-1,1)$.

\bigskip

\par This contradicts the fact that the curve $\mathcal{F}$ has two distinct tangents at its singular points. For example, for the point $(0,0,1)$, by considering the affine curve $\mathcal{Z}(f')$, with
\[f'(X,Y)=X^2Y^2-\mu'(X^2+Y^2),\]
it follows that the curve $\mathcal{F}$ has $2$ distinct tangent lines at $(0,0,1)$
since $q$ is odd.

\bigskip

\par It follows that the curve $\mathcal{F}$ is absolutely irreducible with genus
$$g=\binom{3}{2}-\sum_{P\in Sing(\mathcal{F})}\binom{2}{2}=0.$$
In particular, $\mathcal F$ has $q+1$ $\F$-rational points. By the arguments used at the start of the proof this implies that the point $V(a,0,c)$ is $H$-covered.
\end{proof}

As we have seen in case \ref{(i)}, there is always a point with coordinates $(a,b,0)$ which is $H$-free. To get a complete arc, we now define the point $R_0(a_0,b_0,0)$
where $a_0b_0\in\square$ for $q\equiv 3\;\m$ and $a_0b_0\in\triangle$ for $q\equiv 1\;\m$. Consequently, the point $R_0$ is an internal point to $\C$. Let the arc $K$ be defined as $H\cup \{R_0\}$.

\begin{lemma}\label{Ex:lemma4} If $q\geq 17$ then each point $P'\in H'$ is $K$-covered.
\end{lemma}
\begin{proof}
\par Each secant line through $R_0$ intersects the conic $\C$ in points $P\in H$ and $P'\in H'$, since the point $R_0$ is not $H$-covered. But the arc $K$ contains both $H$ and $R_0$, so the points $P'\in H'$ are $K$-covered.
\end{proof}
\begin{lemma}\label{Ex:lemma5} If $q\equiv 1 \;\m$, then the point $(0,0,1)$ is $H$-covered and if $q\equiv 3 \;\m$, then the point $(0,0,1)$ is $K$-covered, if $\frac{b_0}{a_0}\in \{u^4 : u\in \F\setminus \{0\}\}$.
\end{lemma}
\begin{proof} The proof of the first part of the lemma follows from \ref{(ii)}. For the second part of the lemma, we need to show that there is a chord of $K$ through the point $(0,0,1)$. The line passing through the points $(1,\alpha^{4k},\alpha^{2k})$ and $R_0$ contains $(0,0,1)$ if and only if $b_0a_0^{-1}=\alpha^{4k}$.
\end{proof}

\begin{theorem} If $q\geq 17$ and $R_0(a_0,b_0,0)$, with $a_0b_0\in\triangle$ for $q\equiv 1\;\m$ and $\frac{b_0}{a_0}=u^{4}$ for some $u \in \F \setminus \{0\}$ for $q\equiv 3\;\m$,
then the arc $K=H\cup \{R_0\}$ is complete.
\end{theorem}
\begin{proof}
By Lemma~\ref{involution} the point $R(a,b,c)\notin \C$, where $abc\neq 0$, is $H$-covered. The points $U(0,b,c)$ and $V(a,0,c)$, where $bc \neq 0$, $ac\neq 0$ are $H$-covered by Lemma~\ref{Ex:cases}. The point $P'\in H'$ is $K$-covered by Lemma~\ref{Ex:lemma4}. If $q\equiv 1\;\m$, then the point $(0,0,1)$ is $H$-covered, and if $q\equiv 3\;\m$, then the point $(0,0,1)$ is $K$-covered by Lemma~\ref{Ex:lemma5}.
\end{proof}

\section{On Pellegrino's condition}

The complete arc $K=H\cup\{R\}$ of size $(q+3)/2$ contains $(q+1)/2$ points from the conic $\C$ and the internal $H$-free points are on the line $Z=0$. But the line $Z=0$ is not an external line to the conic $\C$. Hence, we have the following corollary.
\begin{corollary}\label{H-free} For $q\geq17$, if $q\equiv 1 \;\m$, then the only H-free points are internal points on the line $Z=0$ and if $q\equiv 3 \;\m$, then the H-free points are internal points on the line $Z=0$ and the point $(0,0,1)$.
\end{corollary}
\begin{remark}
Pellegrino proved that if $H$ is a set of $(q+1)/2$ points on a conic $\C$, satisfying the condition that there exists an external line $\ell$ to $\C$ containing two internal $H$-free points $P_1$, $P_2$, then for $q>13$:
\begin{enumerate}[label=(\roman*)]
\item when $q\equiv 3\;\m$ , the arc $K=H\cup\{P_1,P_2\}$ is complete;
\item when $q\equiv 1\;\m$ , the arcs $K=H\cup \{P_1,P_2\}$ and $K'=H\cup\{P\}$ are complete where $P$ is the pole of $\ell$ with respect to a conic $\C$.
\end{enumerate}
However, Corollory~\ref{H-free} shows that Pellegrino's condition from \cite{Pellegrino1993}, ``precisely, that
there exists an external line containing two internal $H$-free points" is not always satisfied.  This gives a counterexample to the claim in \cite{Hirschfeld_review}.
\end{remark}

\section{Final comments}
Using the same construction, for small values of $q$ computations using the GAP~\cite{GAP}-package and FinInG~\cite{fining} show that if $q=9$,
then we should add $3$ internal $H$-free points to the set $H$ to obtain a complete arc $K$ from the set $H$. Then $|K|=(q+7)/2=8$, and it is known that there is a unique such complete arc in $\mathrm{PG}(2,9)$. If $q=11$ or $13$, then we should add $2$ internal $H$-free points to the set $H$ giving the complete arcs of size either $8$ or $9$.
\bigskip
\par In summary, if the point $R$ is an external point of $\C$ and roughly half of the points of the non-degenerate conic is taken, then there are complete arcs of sizes at least $(q+5)/2$ and at most $(q+11)/2$, and if the point $R$ is an internal point of $\C$, then there are complete arcs of sizes at least $(q+3)/2$ and at most $(q+5)/2$.

\section{Acknowledgements}
The authors would like to thank Tam\'as Sz\H{o}nyi and G\'abor Korchm\'aros for their valuable discussions and contributions to the paper. We thank the anonymous reviewers for their thoughtful comments and suggestions, and for pointing out Remark~\ref{remark}.

\bibliographystyle{amsplain}

\end{document}